\documentclass[12pt,a4paper]{article}

\usepackage[latin1]{inputenc}
\usepackage{graphicx}
\usepackage{amssymb, amsmath}
\usepackage{color}
\usepackage{geometry}
\usepackage{enumitem}
\usepackage{amsthm}

\usepackage{amsfonts,amsthm,amssymb,amsmath}

\usepackage[dvipsnames]{xcolor}
\usepackage[colorlinks=true,linkcolor=blue,urlcolor=blue,citecolor=blue]{
hyperref}
\newtheorem{theorem}{Theorem}[section]
\newtheorem{definition}[theorem]{Definition}
\newtheorem{lemma}[theorem]{Lemma}
\newtheorem{corollary}[theorem]{Corollary}
\newtheorem{proposition}[theorem]{Proposition}

\newtheorem{example}[theorem]{Example}
\newtheorem{remark}[theorem]{Remark}

\newcommand{\T}{\mathbb{T}}

\newcommand{\N}{\mathbb{N}}
\newcommand{\C}{\mathbb{C}}
\newcommand{\Ha}{\mathcal{H}}

\usepackage{amsmath,amsfonts,amssymb,amsthm}
\usepackage{mathtools}

\let\norm\undefined 
\DeclarePairedDelimiter\norm{\lVert}{\rVert}

\begin{document}
\title{Random unconditional convergence of vector-valued Dirichlet series}
\author{Daniel Carando, Felipe Marceca, Melisa Scotti, Pedro Tradacete}

\date{}


\renewcommand{\thefootnote}{\fnsymbol{footnote}}



\renewcommand{\thefootnote}{\arabic{footnote}}

\maketitle

\begin{abstract}
	We study random unconditionality of Dirichlet series in vector-valued Hardy spaces $\mathcal H_p(X)$. It is shown that a Banach space $X$ has type 2 (respectively, cotype 2) if and only if for every choice $(x_n)_n\subset X$ it follows that $(x_n n^{-s})_n$ is Random unconditionally convergent (respectively, divergent) in $\mathcal H_2(X)$. The analogous question on $\mathcal H_p(X)$ spaces for $p\neq2$ is also explored. We also provide explicit examples exhibiting the differences between the unconditionality of $(x_n n^{-s})_n$ in $\mathcal H_p(X)$ and that of $(x_n z^n)_n$ in $H_p(X)$.
\end{abstract}

\thispagestyle{empty}

\section{Introduction}
	In this article we investigate some basic questions about random unconditionality of Dirichlet series in vector-valued Hardy spaces.
Given a complex Banach space $X$, a Dirichlet series in $X$  is a series of the form $D=\sum_n x_n n^{-s}$, where the coefficients $x_n$ are vectors in $X$ and $s$ is a complex variable.
The study of functional-analytic aspects of the theory of (vector-valued) Dirichlet series has attracted great attention in the recent years (see, for example, \cite{CaDeSe14,CaDeSe1,DeGaMaPG08,H+,DeSchSP14}, and also \cite{CMGaMa18}, where the vector-valued theory is used to study multiple Dirichlet series).
The Hardy space $\mathcal H_p(X)$ of $X$-valued Dirichlet series consists, loosely speaking, of those Dirichlet series  whose corresponding Hardy $p$-norm via Bohr's transform is finite (see next section for the formal definition). It is well known that, even in the scalar case, the standard basis $(n^{-s})_n$ of the space $\mathcal H_p(\mathbb C)$ is unconditional only when $p=2$ (see \cite[Proposition 4]{CaDeSe1}). 
	Since unconditionality is hard to accomplish, we are lead to consider weaker versions, such as \emph{random unconditionality}. While unconditional convergence of a series $\sum_n y_n$ is equivalent to the convergence of $\sum_n \varepsilon_n y_n$ for all choice of signs $\varepsilon_n=\pm 1$,
 random unconditional convergence is related to the convergence of  $\sum_n \varepsilon_n x_n$ for \emph{almost every choice} of signs $(\varepsilon_n)_n\in \{-1,+1\}^\mathbb N$ (with respect to Haar measure).

We are interested in identifying Random Unconditional Convergent (in short, RUC, see \cite{SAM}) and Random Unconditional Divergent (in short, RUD, see \cite{lopez2016bases}) systems of vector-valued Dirichlet series. Namely, a sequence $(x_n)_n$ (usually part of a biorthogonal system) in a Banach space is called RUC when there is a uniform estimate of the form
$$
\mathbb E \Big\| \sum_n \varepsilon_n a_n x_n\Big\| \lesssim \Big\|\sum_n a_n x_n\Big\|,
$$
for every choice of scalars $(a_n)_n$, where $\mathbb E$ denotes the expectation with respect to i.i.d. Rademacher random variables $\varepsilon_n$; analogously, $(x_n)_n$ is called RUD when the converse estimate holds
$$
\Big\|\sum_n a_n x_n\Big\| \lesssim \mathbb E \Big\|\sum_n \varepsilon_n a_n x_n\Big\|.
$$
With this terminology, one can deduce from \cite[Proposition 4]{CaDeSe1} that the canonical basis $(n^{-s})_{n}\subset \mathcal H_p(\mathbb C)$ is RUC, for $p\geq2$, while it is RUD for $p\leq2$. The main question we want to address here corresponds to the vector-valued version of this phenomenon: When is $(x_n n^{-s})_n$ a RUC system in $\mathcal H_p(X)$ (respectively, RUD), for every choice of $(x_n)_n\subset X$?

These questions have the following equivalent formulations. Suppose a Dirichlet series $D=\sum_n x_n n^{-s}$ belongs to $\mathcal H_p(X)$; does $\sum_n \varepsilon_n x_n n^{-s}$ also belong to $\mathcal H_p(X)$ for almost every choice of signs $(\varepsilon_n)_n\in\{-1,+1\}^\mathbb N$? In the opposite direction, if $\sum_n \varepsilon_n x_n n^{-s}$ belongs to $\mathcal H_p(X)$ for almost every choice of signs $(\varepsilon_n)_n\in\{-1,+1\}^\mathbb N$, does $\sum_n  x_n n^{-s}$ necessarily belong to $\mathcal H_p(X)$? This notion of almost sure sign convergence of Dirichlet series has been the object of recent research in \cite{CaDeSe1}, where the following space was introduced
\begin{equation} \label{hprad-original}
\mathcal H^{rad}_p(X) := \big\{ \textstyle\sum_n x_{n} n^{-s} \colon \,  \sum_n \varepsilon_{n} x_{n} n^{-s} \in \mathcal H_p(X)\, \text{for a.e. } (\varepsilon_{n})_n\in\{-1,+1\}^\mathbb N\big\} \,.
\end{equation}
The previous questions can be reformulated in terms of inclusion relations between $\mathcal H_p(X)$ and $\mathcal H^{rad}_p(X)$. One of our main results in this direction is the following characterization (see Theorem~\ref{T2implicaH2RCP}):

\bigskip
\noindent\textbf{Theorem A} For a Banach space $X$ the following statements are equivalent:
\begin{enumerate}[label=\rm{(\alph*)}]
	\item $(x_n n^{-s})_n$ is RUC in $\mathcal H_2(X)$, for every choice of $(x_n)_n\subset X$.
	\item $\mathcal H_2(X)\subset \mathcal H^{rad}_2(X)$.
	\item $X$ has type 2.
\end{enumerate}
\bigskip

In an analogous way, one can prove that the spaces $X$ where $\mathcal H^{rad}_2(X)\subset \mathcal H_2(X)$, or where $(x_n n^{-s})_n$ is always RUD, are precisely those with cotype 2. The case when $p\ne 2$ will also be considered in Section \ref{sec-typecotype}.

The paper is structured as follows: in the next section we introduce preliminaries and notation on Dirichlet series and random unconditionality. Section \ref{sec-HpRCP} is devoted to provide equivalent reformulations of the property that a Banach space $X$ satisfies that $(x_n n^{-s})_n$ is a RUC system in $\mathcal H_p(X)$ (respectively, RUD), for every choice of $(x_n)_n\subset X$. As an application of Green-Tao's theorem on arithmetic progressions in the set of primes \cite{TEOTAO}, we will provide examples that show how differently Dirichlet series and power series behave in terms of random unconditionality. In particular, we will show that for a fixed sequence of vectors $(x_n)_n$, one could have that $(x_n n^{-s})_n$ is RUC in $\mathcal H_p(X)$, while $(x_n z^n)_n$ is not RUC in $H_p(X)$, and vice versa. Finally, in Section~\ref{sec-typecotype}, we provide the connection between type (or cotype) and the above questions.


%
%

\section{Definitions and general results}

We refer to the books \cite{DGMSxx} and \cite{QueQue13} for the general theory of Dirichlet series.
We denote by $\mathbb T^{\mathbb N}$ the infinite complex polytorus
$$
\mathbb T^{\mathbb N}=\{(z_n)_{n\in\mathbb N}:z_n\in \mathbb C, |z_n|=1\}.
$$
Given $1\leq p<\infty$ and a Banach space $X$, let $L_p(\mathbb T^{\mathbb N},X)$ be the space of
$p$-Bochner integrable functions $f:\mathbb T^{\mathbb N}\rightarrow X$ with respect to the Haar measure. Let us write $\mathbb N_0=\mathbb N\cup\{0\}$ and denote by $\mathbb Z^{(\mathbb N)}$ (respectively $\mathbb N_0^{(\mathbb N)}$ )  the set of eventually null sequences of integer numbers (respectively, non-negative integer numbers). Also, for any sequence of scalars $z=(z_n)_{n}$ and $\alpha=(\alpha_1,\ldots,\alpha_m,0\ldots)\in\mathbb Z^{(\mathbb N)}$, let us denote $z^\alpha=z_1^{\alpha_1}\cdot \ldots \cdot z_m^{\alpha_m}$.

Recall that every $f\in L_p(\mathbb T^{\mathbb N},X)$ is uniquely
determined by its (formal) Fourier series
$$
\sum_{\alpha\in\mathbb Z^{(\mathbb N)}} \hat f(\alpha)z^\alpha,
$$
where
$$
\hat f(\alpha)=\int_{\mathbb T^{\mathbb N}}f(w)w^{-\alpha}dw.
$$
The space $H_p(\mathbb T^{\mathbb N},X)$ is the closed subspace of $L_p(\mathbb
T^{\mathbb N},X)$ consisting of those functions $f$ with $\hat f(\alpha)=0$
whenever $\alpha\in \mathbb Z^{(\mathbb N)}\backslash \mathbb N_0^{(\mathbb
N)}$.

Given $\alpha=(\alpha_1,\ldots,\alpha_m,0\ldots)\in\mathbb N_0^{(\mathbb N)}$,
let us denote $n(\alpha)=p_1^{\alpha_1}\cdot\dots\cdot p_m^{\alpha_m}\in \mathbb N$, where
$p_1<p_2<\dots$ is the ordered sequence of prime numbers. Similarly, given a natural number
$n=p_1^{\alpha_1}\cdot\dots\cdot p_m^{\alpha_m}$, define
$\alpha(n)=(\alpha_1,\ldots,\alpha_m,0\ldots)\in\mathbb N_0^{(\mathbb N)}$.

We can formally consider Bohr's transform
$$
\mathcal B \bigg(\sum_{\alpha \in\mathbb N_0^{(\mathbb N)}} x_\alpha
z^\alpha\bigg)=\sum_{n\in\mathbb N} x_{\alpha(n)}n^{-s},
$$
and define $\mathcal H_p(X)$ as the image of $H_p(\mathbb T^{\mathbb N},X)$
equipped with the norm that turns this mapping into an
isometry. To be more precise, a Dirichlet series $D=\sum_n x_n n^{-s}$ is in $\mathcal H_p(X)$ if there is a function $f\in H_p(\mathbb T^{\mathbb N},X)$ such that $\hat f(\alpha)=x_{n(\alpha)}$ and in that case
\[\|D\|_{\mathcal H_p(X)}=\|f\|_{H_p(\mathbb T^{\mathbb N},X)}.\]
In particular, if $(x_n)_n\subseteq X$ has finitely many non zero elements, we have
$$
\Big\|\sum_{n\in\mathbb N} x_n n^{-s}\Big\|_{\mathcal
H_p(X)}=\Big\|\sum_{\alpha\in\mathbb N_0^{(\mathbb N)}} x_{n(\alpha)}
z^\alpha\Big\|_{H_p(\mathbb T^{\mathbb N},X)}.
$$
Due to this isometry some properties are translated from the power series to the Dirichlet series setting.

\begin{enumerate}[label=(\roman{*})]
	\item \label{coeff} For a Dirichlet series $D \in \Ha_p(X)$ the coefficients of D are bounded by $\|D \|_{\Ha_p(X)}$. More precisely, the operator $c_n$ that takes the $n-$th coefficient is contractive. As a consequence, if a sequence of Dirichlet series $(D_N)_N$ converges in $\mathcal{H}_p(X)$ to some $D$, the coefficients $c_n (D_N)$ converge to $c_n(D)$ for all $n \in \mathbb{N}$.
	\item \label{dense} The set of Dirichlet polynomials $\sum_{n=1}^N x_n n^{-s}$ is dense in $\mathcal{H}_p (X)$ for $1 \le p < \infty$ since the analytic polynomials are dense in $H_p (\mathbb{T}^\N, X) $ \cite[Proposition 24.6]{DGMSxx}.
\end{enumerate}

Recall that a basis $(x_n)_n$ of a Banach space $X$ is unconditional if for every $x\in X$, its expansion $\sum_n a_n x_n$ converges unconditionally. Equivalently, there is a constant $C>0$ such that for every $m \in\mathbb N$ and every sequence of scalars $(a_n)_{n=1}^m$, we have
\begin{equation}\label{unconditional}
\sup_{\varepsilon_n\in\{-1,+1\}} \Big\|\sum_{n=1}^m \varepsilon_n a_n x_n\Big\| \leq
C\Big\|\sum_{n=1}^m a_n x_n\Big\|.
\end{equation}


Banach spaces with unconditional bases have a nice structure, including a wealth of operators acting on them. However, in the landmark paper \cite{GM}, Banach spaces which do not have any subspace with an unconditional basis are constructed. Therefore, weaker versions of unconditionality have to be considered. In this direction, we will next discuss two notions of random unconditionality that were introduced in \cite{SAM} and \cite{lopez2016bases}.

Recall, a series $\sum_n y_n$ in a Banach space is random unconditionally convergent when $\sum_n \varepsilon_n y_n$ converges almost surely on signs $(\varepsilon_n)\in\{-1,+1\}^\mathbb N$ with respect to Haar probability measure on $\{-1,+1\}^\mathbb N$. A basis $(x_n)_n$ of a Banach space $X$ is of Random Unconditional convergence (in short, RUC), if every convergent series $\sum_n \varepsilon_n
 x_n$ is random unconditionally convergent. Analogously, we say $(x_n)_n$ is a basis of Random Unconditional divergence (in short, RUD), if every random unconditionally convergent  series $\sum_n  a_n x_n$ must be convergent.

An equivalent formulation of these notions can be given in terms of the expectation
$$
\mathbb E \Big\|\sum_{n=1}^m \varepsilon_n y_n\Big\|=\frac{1}{2^m}\sum_{(\varepsilon_n)\in\{-1,+1\}^m} \Big\|\sum_{n=1}^m \varepsilon_n y_n\Big\|.
$$
Indeed, $(x_n)_n$ is RUC if and only if there is a constant $C$ such that for every
$m \in \N$ and every sequence of scalars $(a_n)_{n=1}^m$ one has that
\begin{equation}\label{RUC}
\mathbb E \Big\|\sum_{n=1}^m \varepsilon_n a_n x_n\Big\| \leq
C\Big\|\sum_{n=1}^m a_n x_n\Big\|.
\end{equation}
In this case, we will say that $(x_n)_n$ is $C$-RUC. Similarly, $(x_n)_n$ is RUD if and only if there is a constant $C$ such that for every $m \in \N$ and
every sequence of scalars $(a_n)_{n=1}^m$ one has that
\begin{equation}\label{RUD}
\Big\|\sum_{n=1}^m a_n x_n\Big\| \le C\mathbb E \Big\|\sum_{n=1}^m
\varepsilon_n a_n x_n\Big\| .
\end{equation}
In this case, we will say that $(x_n)_n$ is $C$-RUD. It is immediate to see that a basis $(x_n)_n$ is unconditional if and only if it is both RUC and RUD (see~\cite[Proposition 2.3]{lopez2016bases}).

Moreover, the notions of RUC and RUD make also sense in the more general context of biorthogonal systems. Notice that for every sequence $(x_n)_{n}\subseteq X$, the non-zero elements of $(x_n n^{-s})_{n}$ can be considered as part of a biorthogonal system. Indeed, write
$\left( x_n n^{-s} \right)_{n\in A} \in \mathcal{H}_p(X)$  where $A \subseteq \N$ is the set of indexes $n$ for which  $x_n\neq 0$. Applying the Bohr transform we may regard this sequence as $\left( x_{\alpha}z^{\alpha} \right)_{\alpha \subseteq \Lambda} \in H_p (X) $ for a suitable $\Lambda$ corresponding to $A$. For each $\alpha\in \Lambda$ choose $\gamma_{\alpha} \in X^*$ such that $\gamma_{\alpha}(x_{\alpha})=1$. Notice that if we define $\varphi_{\alpha} \in H_p(X)^* $ by
$$
\varphi_{\alpha} (f)= \int_{\mathbb{T}^{\mathbb{N}}} \gamma_{\alpha}\left(f(z) \right) \overline{z}^{\alpha} dz,
$$
then $\left(x_{\alpha}z^{\alpha} , \varphi_{\alpha} \right)_{\alpha\in\Lambda}$ is a biorthogonal system, since
$$
\varphi_{\beta} (x_{\alpha} z^{\alpha}) = \int_{\mathbb{T}^{\mathbb{N}}} \gamma_{\alpha}(x_{\alpha}) z^{\beta} \overline{z}^{\alpha} dz = \delta_{\alpha , \beta}.
$$
From now on, when we say that $(x_n n^{-s})_{n}$ is RUC or RUD we mean that the nonzero elements are RUC or RUD as part of the biorthogonal system just defined. Note that the conditions \eqref{RUC} and \eqref{RUD} which define RUC and RUD can be checked for the whole sequence (it is not necessary to omit the zero elements).

When dealing with expectations of the form $\mathbb E\|\sum_n \varepsilon_n x_n\|$, we will repeatedly make use of \emph{Kahane's inequality} (cf. \cite[11.1]{DJT}): For any $0<p<\infty$, there is $K_p>0$ such that for every $(x_n)_{n=1}^m$ in a Banach space
\begin{equation}\label{kahane}
\frac{1}{K_p}\bigg(\mathbb E\Big\|\sum_{n=1}^m \varepsilon_n x_n\Big\|^p\bigg)^{\frac1p}\leq \mathbb E\Big\|\sum_{n=1}^m \varepsilon_n x_n\Big\|\leq K_p\bigg(\mathbb E\Big\|\sum_{n=1}^m \varepsilon_n x_n\Big\|^p\bigg)^{\frac1p}.
\end{equation}

Another fundamental property that will be used throughout  is the \emph{Contraction principle} (cf. \cite[12.2]{DJT}, see also \cite{Seigner} for the sharp version for complex scalars).
 For any scalars $(a_n)_{n=1}^m\subset \mathbb C$ and any $(x_n)_{n=1}^m$ in a Banach space
\begin{equation}\label{contraction}
\mathbb E\Big\|\sum_{n=1}^m \varepsilon_n a_n x_n\Big\|\leq \frac{\pi}{2} \max|a_n| \mathbb E\Big\|\sum_{n=1}^m \varepsilon_n x_n\Big\|.
\end{equation}

Moreover, since Steinhaus variables are symmetric $(z_n)_{n=1}^m$ and $(\varepsilon_n z_n)_{n=1}^m$ are identically distributed. Therefore, we also have

\begin{equation}\label{contraction2}
\mathbb E\Big\|\sum_{n=1}^m  a_n x_n z_n \Big\|\leq \frac{\pi}{2} \max|a_n| \mathbb E\Big\|\sum_{n=1}^m x_n z_n\Big\|.
\end{equation}

\begin{remark}
	If $(x_n)_n$  is $RUC$ in $X$ then the sequence $(x_n n^{-s})_n$  is $RUC$ in $\mathcal{H}_p (X)$ for every $p \geq 1$.
\end{remark}
\begin{proof}
By Kahane's inequality \eqref{kahane}, and applying the $RUC$ condition to $(x_n)_n$ with coefficients  $a_n z^{\alpha(n)}$ (for a fixed $z\in \T^\N$), we obtain
$$
\mathbb{E} \Big\| \sum_{n=1}^m \varepsilon_n a_n x_n z^{\alpha(n)} \Big\|^p_{X} \lesssim \Big\| \sum_{n=1}^m a_n x_n z^{\alpha(n)} \Big\|^p_{X}.
$$
Integrating with respect to $z$ and switching the order of integration, again by Kahane's inequality, the last expression becomes
\[
\mathbb{E} \Big\| \sum_{n=1}^m \varepsilon_n a_n x_n z^{\alpha(n)} \Big\|_{H_p (\T^\N,X)} \lesssim \Big\| \sum_{n=1}^m a_n x_n z^{\alpha(n)} \Big\|_{H_p (\T^\N,X)}.
\]
Equivalently applying Bohr's transform we get
\[ \mathbb{E} \Big\| \sum_{n=1}^m \varepsilon_n a_n x_n n^{-s} \Big\|_{\mathcal H_p (X)} \lesssim \Big\| \sum_{n=1}^m a_n x_n n^{-s} \Big\|_{\mathcal H_p (X)}.\]
\end{proof}

It  is interesting to observe that the converse is not true.  A simple example of a sequence which is neither $RUC$ nor $RUD$  is the \textit{summing basis} in $c_0$ (see \cite{lopez2016bases}), defined by
\[
s_n= \sum_{i=1}^n e_i  \qquad n \in \mathbb{N},
\]
where $(e_i)$ denotes the canonical basis of $c_0$. However, the sequence $(s_n n^{-s})_n$ is actually  equivalent to the $\ell_2$ basis and thus unconditional. To see this, we recall the following property of the summing basis:
$$ \Big\|\sum_{n=1}^m a_n s_n\Big\|=\sup_{1\le k \le m} \Big| \sum_{n=k}^{m} a_n \Big|.$$

Hence, it follows that
\begin{align*}
\Big\|\sum_{n=1}^m a_n s_n n^{-s}\Big\|_{\Ha_2 (c_0)}= \bigg( \int  \sup\limits_{1 \le k \le m} \Big| \sum_{n=k}^{m} a_n n^{-s}\Big|^2 ds\bigg)^{1/2} =\Big\| \sup_{1 \le k \le m} \Big| \sum_{n=k}^m a_n n^{-s} \Big| \Big\|_{\Ha_2(\mathbb C) }.
\end{align*}
On the one hand, we obviously have
\begin{equation}\label{lower2}
\Big\| \sup_{1 \le k \le m} \Big| \sum_{n=k}^m a_n n^{-s} \Big| \Big\|_{\Ha_2(\mathbb C) }\geq \Big\| \sum_{n=k}^m a_n n^{-s}  \Big\|_{\Ha_2(\mathbb C) }
= \bigg( \sum_{n=1}^{m} \left| a_ n\right|^2  \bigg)^{1/2}.
\end{equation}

On the other hand, a version of Carleson-Hunt's theorem for Dirichlet series  (see the proof of Theorem 1.5 in \cite{hedenmalm2003carleson}) 
provides us with $C>0$ such that
\begin{equation*}
\Big\| \sup_{1 \le k \le m} \Big| \sum_{n=k}^m a_n n^{-s} \Big| \Big\|_{\Ha_2(\mathbb C) } \le  C  \Big\|\sum_{n=1}^m a_n n^{-s}\Big\|_{\Ha_2(\mathbb C)} = C  \bigg(  \sum_{n=1}^m \left| a_ n\right|^2  \bigg)^{1/2}.
\end{equation*}
Joining this with \eqref{lower2} yields the desired result.

\subsection{The Banach space $\mathcal{H}^{rad}_p (X) $}
In order to study the almost sure convergence of random Dirichlet series, we recall the space defined in \cite{CaDeSe1}:
\begin{definition}
	Given $1\le p \le \infty $ and a Banach space $X$, we define
	\begin{equation*}
			\Ha_p ^{rad}(X):= \left\{ \sum x_n n^{-s} : \forall \textrm{ a.e.
		} \varepsilon_n = \pm 1 ,  \sum \varepsilon_n x_n n^{-s}  \in \mathcal{H}_p (X) \right\}.
	\end{equation*}
\end{definition}

A Dirichlet series  $ \sum_n x_n n^{-s}$ should be regarded as a formal expression. As it was mentioned before, when we say $D=\sum_n x_n n^{-s} \in \mathcal{H}_p (X)$ we mean that there is a function $f\in H_p(\T^\N,X)$ with Fourier coefficients $\hat f_\alpha= x_{n(\alpha)}$. Therefore, the fact that $D\in \mathcal{H}_p (X)$ does not necessarily imply that the partial sums $D_N=\sum_{n=1}^N x_n n^{-s}$ converge to $D$ in $\mathcal{H}_p (X)$. Luckily, for random Dirichlet series we have the following result.

%
%
%
%
%
%
\begin{proposition}\label{seryestar} For every Banach space $X$ and every $1 \le p < \infty$. If
 $(x_n)_{n} \subseteq X$ satisfies that $\sum_n \varepsilon_n x_n n^{-s}  \in \mathcal{H}_p (X)$ for almost every choice of signs $\varepsilon_n$ then $\sum_n \varepsilon_n x_n n^{-s}$ converges a.e.
\end{proposition}

To accomplish this we need the following classical theorem which can be found in \cite[Theorem 2.1.1]{KW}.

\begin{theorem}[It\^o-Nisio]\label{ito}
	Let $Y_1, Y_2, \ldots$ be a sequence of independent symmetric random variables with values in a separable Banach space $Y$. Then the following conditions are equivalent:
	\begin{enumerate}[label=\rm{(\alph*)}]
		\item $ \sum_{n=1}^{\infty} Y_n$ converges a.e.;
		\item	There exists a random variable $R$ with values in $Y$ and a family $\mathcal{F}\subseteq Y'$ separating points in $Y$, such that for each $y'$ in $\mathcal F$ the series $ \sum_{n=1}^{\infty} y'(Y_n)$ converges a.e. to $y'(R)$.
	\end{enumerate}
\end{theorem}	
\begin{proof}[Proof of Proposition \ref{seryestar}]

Let $(x_n)_{n} \subseteq X$ be as stated and for $\varepsilon=(\varepsilon_n)_{n}\in\{-1,+1\}^\mathbb N$ set $R(\varepsilon)=\sum_n \varepsilon_n x_n n^{-s}$. We start by showing that $R$ is a random variable (i.e. a measurable function). For $\sigma >0$ and a Dirichlet series $D= \sum_n a_n n^{-s} \in \Ha_p (X)$ define $D_{\sigma}= \sum_n \left( a_n / n^{\sigma}\right) n^{-s} $. In \cite[Proposition 2.3]{H+} it is shown that $D_\sigma$ converges to $D$ in $\Ha_p (X)$ as $\sigma\to 0$, and the partial sums of $D_{\sigma}$ converge to $D_\sigma$ uniformly. Applying this to $R$ we may construct a sequence of measurable functions converging almost everywhere to $R$. Furthermore, the same argument shows that $R(\varepsilon)$ belongs to ${Y=\overline{\langle x_n n^{-s} \rangle_n} \subseteq \Ha_p(X)}$ for almost every $\varepsilon$.
Set
$Y_n = \varepsilon_n x_n n^{-s} \in \Ha _p (X)$ and for every $k \in \N$ and $x' \in X'$ define $\varphi_{k, x'}= x'\circ c_k$ where $c_k$ (as defined in \ref{coeff}) returns the $k-$th coefficient of a Dirichlet series. Notice that the family $\mathcal{F}=\{\varphi_{k, x'}\}_{k,x'}$ separates points of $Y$. Furthermore, we have
\begin{align*}
	 \varphi_{k, x'}(R(\varepsilon))=  \varepsilon_k x' (x_k) \textit{ and }\,
	 \sum_{n=1}^N \varphi_{k,x'}(Y_n) = \sum_{n=1}^N \delta_{n,k} \varepsilon_k x'(x_k),
\end{align*}
and both coincide when $N\geq k$. This means that assertion $(b)$ of the theorem holds. Thus, the partial sums of $R(\varepsilon)$ converge for almost every choice of signs $\varepsilon_n$.
\end{proof}

\begin{corollary}\label{seryestarcoro} For every Banach space $X$ and every $1 \le p < \infty$
	\begin{equation*}\label{definition RAD P}
	\mathcal{H}^{rad}_p (X) = \left\{  \sum x_n n^{-s} : \forall \textrm{ a.e.
	} \varepsilon_n = \pm 1 ,  \sum \varepsilon_n x_n n^{-s} \textit{ converges in }
	\mathcal{H}_p (X) \right\}.
	\end{equation*}
\end{corollary}

Next we are going to endow $\mathcal{H}^{rad}_p (X)$ with a norm. To do this we apply \cite[Proposition 12.3]{DJT} to obtain the following reformulation

\begin{equation*}\label{alternative definition RAD P}
\mathcal{H}^{rad}_p (X) = \left\{  \sum x_n n^{-s} :  \sum r_n
x_n n^{-s} \in L_1 \left([0,1], \mathcal{H}_p (X) \right) \ \right\},
\end{equation*}
where $r_n$ denote Rademacher random variables.
Hence it is natural to define

\begin{equation*}\label{norm in RAD P}
\Big\| \sum x_n  n^{-s} \Big\|_{\mathcal{H}^{rad}_p (X)} := \int_0^1
\Big\| \sum r_n (t) x_n  n^{-s} \Big\|_{\mathcal{H}_p (X)} dt,
\end{equation*}
which turns $\mathcal{H}^{rad}_p (X)$ into a Banach space (see \cite{CaDeSe1}).

Recall that the space of
unconditionally summable sequences $(x_n)_n$ in a Banach space is denoted
\begin{equation*}\label{definition RAD(X)}
Rad(X)
\end{equation*}
and becomes a Banach space under the norm

\begin{equation*}
\left\| (x_n)_n \right\|_{Rad(X)} := \int_0^1 \Big\| \sum r_n(t) x_n \Big\| \,dt,
\end{equation*}
(cf. \cite[Chapter 12]{DJT}). Note that by Kahane's inequalities \eqref{kahane} we may replace the right term by any $p$-norm with $1 \le p < \infty$ to get an equivalent norm.

The next proposition provides an easy way to estimate $\Ha_p^{rad}$ norms without computing Dirichlet norms.

\begin{proposition}\label{rad}
For $1 \le p < \infty $ there is $C_p>0$ such that if $X$ is a Banach space and $(x_n)_{n=1}^N\subset X$ we have
$$
\frac{1}{C_p}\Big\| \sum_{n=1}^N x_n n^{-s} \Big\|_{\mathcal{H}^{rad}_{p} (X)}\leq \left\| (x_n)_n \right\|_{Rad(X)}\leq C_p\Big\| \sum_{n=1}^N x_n n^{-s} \Big\|_{\mathcal{H}^{rad}_{p} (X)},
$$
where we set $x_n=0$ for $n> N$.
In particular, we have that	$\mathcal{H}^{rad}_{p} (X) =  Rad(X)$, up to an equivalent norm.
\end{proposition}

\begin{proof}
Given arbitrary $(x_n)_{n=1}^N\subset X$, using Kahane's inequality \eqref{kahane} and the Contraction principle \eqref{contraction}, we have
	\begin{align*}
	\Big\| \sum_{n=1}^N x_n n^{-s} \Big\|_{\mathcal{H}^{rad}_{p} (X)} &= \int_{0}^{1}  \Big\| \sum_{n=1}^N r_n (t) x_n n^{-s} \Big\|_{\mathcal{H}_{p} (X)} dt \\
	&\sim  \bigg( \int_{0}^{1}  \Big\| \sum_{n=1}^N r_n (t) x_n n^{-s} \Big\|^p_{\mathcal{H}_{p} (X)} dt \bigg)^{1/p}\\ 
	&= \bigg( \int_{\mathbb{T}^{\mathbb{N}}} \int_0^1 \Big\| \sum_{\alpha\in (\mathbb N_0)^\mathbb N} r_{n(\alpha)} (t) x_{n(\alpha)} z^{\alpha} \Big\|^p_{X} dt dz\bigg)^{1/p}\\
	&\sim \bigg( \int_0^1 \Big\| \sum_{\alpha\in (\mathbb N_0)^\mathbb N} r_{n(\alpha)} (t) x_{n(\alpha)} \Big\|^p_{X} dt \bigg)^{1/p}\\ 
	&\sim  \int_0^1 \Big\| \sum_{n=1}^N r_n (t) x_n \Big\|_{X} dt.  \\
	\end{align*}
Clearly, the equivalence constants above only depend on $1 \le p < \infty$.	
\end{proof}

\section{The $ \mathcal{H}_p$ random convergence property}\label{sec-HpRCP}

In this section we will focus on characterizing those Banach spaces $X$ such that every sequence $(x_n)_{n}\subset X$ satisfies that $(x_n n^{-s})_{n}$ is $RUC$ in $ \mathcal{H}_p (X)$. By the definition of a $RUC$ system, this means that there exists a constant (depending on the sequence) such that the inequality \eqref{RUC} is satisfied. Next proposition shows that a uniform constant (not depending on the sequences) can be chosen. In fact, this condition is also equivalent to the inclusion $\Ha_p(X) \subseteq \Ha_p^{rad}(X)$.

\begin{proposition}\label{EqRCP}
	Let $X$ be a Banach space and $p \geq 2$. The following statements are equivalent:
	\begin{enumerate}[label=\rm{(\alph*)}]
		\item $(x_n n^{-s})_{n} $ is RUC in $ \mathcal{H}_p (X)$ for every $(x_n)_{n} \subset X$.
		\item There is $C \ge 1$ such that for every $N \in \mathbb{N}$ and $(x_n)^N_{n=1}$ we have
		\begin{align}
		\label{ec3}
		\Big\|\sum_{n=1}^N x_n n^{-s} \Big\|_{\mathcal{H}^{rad}_p (X)} \leq
		C \Big\|\sum_{n=1}^N x_n n^{-s} \Big\|_{\mathcal{H}_p (X)}.
		\end{align}
		\item The following inclusion holds:
		$$ \mathcal{H}_p (X) \subseteq \mathcal{H}^{rad}_p (X). $$
		\item There is $C\geq 1$ such that for every $N \in \N$ and $(x_n)^N_{n=1}$ we have
		\begin{equation}\label{ec4}
		\mathbb{E} \Big\| \sum_{n=1}^N \varepsilon_n x_n \Big\| \le C \bigg( \int_{\mathbb{T}} \Big\| \sum_{n=1}^N x_n z^n \Big\|^p dz \bigg)^{1/p}.
		\end{equation}
	\end{enumerate}
\end{proposition}	

\begin{definition}
Given $p\geq 2$, we will say that a Banach space $X$ has the $ \mathcal{H}_p$ \textit{random convergence property} (or, in short, $X$ has $\mathcal{H}_p -RCP$ ) if $X$ satisfies any (and all) of the conditions in Proposition \ref{EqRCP}
\end{definition}

In order to prove Proposition \ref{EqRCP}, we need the following lemma.

\begin{lemma}
\label{lema}
Assume that there exist $M\in \mathbb{N}$ and $K>0$ such that $(y_n n^{-s})_{n=M+1}^{\infty}$ is $K-$RUC for every vector sequence $(y_n)_{n=M+1}^{\infty}$. Then $(y_n n^{-s})_{n\in\N}$ is $(M+K+KM)-$RUC for every vector sequence
$(y_n)_{n\in\N}\subseteq X$.

\end{lemma}

\begin{proof}
 Let $(x_n)_{n}$ be a vector sequence in $X$. We start by proving that $(x_n n^{-s})_{n}$ is $(M+K+KM)-$RUC. Given $N\in\N$ and  and $(a_n)_{n=1}^N\subseteq\C$, we have
 \[\Big\|\sum_{n=1}^N a_n x_n n^{-s}\Big\|_{\mathcal{H}_p^{rad}(X)}\leq
 \sum_{n=1}^M\left\| a_n x_n n^{-s}\right\|_{\mathcal{H}_p^{rad}(X)}+\Big\|\sum_{n=M+1}^N a_n x_n n^{-s}\Big\|_{\mathcal{H}_p^{rad}(X)}.\]
Applying \ref{coeff} for $1\leq k\leq N$ we get
 \begin{align*}
     \left\| a_k x_k k^{-s}\right\|_{\mathcal{H}_p^{rad}(X)}
 &= \left\| a_k x_k \right\|_{X}
 = \Big\| c_k \bigg(\sum_{n=1}^N a_n x_n n^{-s}\bigg)\Big\|_{X}
 \\&\leq\Big\| \sum_{n=1}^N a_n x_n n^{-s}\Big\|_{\mathcal{H}_p(X)}.
 \end{align*}
 Therefore, we may deduce
 \[\Big\|\sum_{n=1}^N a_n x_n n^{-s}\Big\|_{\mathcal{H}_p^{rad}(X)}\leq
 M\Big\| \sum_{n=1}^N a_n x_n n^{-s}\Big\|_{\mathcal{H}_p(X)}+\Big\|\sum_{n=M+1}^N a_n x_n n^{-s}\Big\|_{\mathcal{H}_p^{rad}(X)}.\]
 On the other hand, since $(x_{n} n^{-s})_{n=M+1}^\infty$ is $K-$RUC by hypothesis, we get
 \begin{align*}
 \Big\|\sum_{n=M+1}^N a_n x_n n^{-s}\Big\|_{\mathcal{H}_p^{rad}(X)}
 &\leq K\Big\|\sum_{n=M+1}^N a_n x_n n^{-s}\Big\|_{\mathcal{H}_p(X)}
 \\ & \leq K\Big\|\sum_{n=1}^N a_n x_n n^{-s}\Big\|_{\mathcal{H}_p(X)}
 +K\Big\|\sum_{n=1}^M a_n x_n n^{-s}\Big\|_{\mathcal{H}_p(X)}
 \\ &\leq K\Big\|\sum_{n=1}^N a_n x_n n^{-s}\Big\|_{\mathcal{H}_p(X)}
 +K\sum_{n=1}^M\left\| a_n x_n n^{-s}\right\|_{\mathcal{H}_p(X)}
 \\ &\leq (K+KM)\Big\|\sum_{n=1}^N a_n x_n n^{-s}\Big\|_{\mathcal{H}_p(X)}.
 \end{align*}
 Joining the last two inequalities concludes the argument.
\end{proof}

\bigskip

\begin{proof}[Proof of Proposition \ref{EqRCP}]
For convenience we will prove first the equivalence $(b)\Leftrightarrow (d)$, and then we will prove $(a)\Rightarrow (b)\Rightarrow (c)\Rightarrow (a)$.

$(b)\Rightarrow(d)$:  Given $(x_n)_n \subseteq X$ we define

$$
\widetilde{x}_k  = \left\{
\begin{array}{c l}
 x_n &  k= 2^n \\
 \quad \\
  0 & k \not= 2^n\\
\end{array}
\right.
$$
Using Proposition \ref{rad} and \eqref{ec3} we have
$$	
	\mathbb{E} \Big\| \sum_{n=1}^N \varepsilon_n x_n \Big\| = \mathbb{E} \Big\| \sum_{k=1}^{2^N} \varepsilon_k \widetilde{x}_k \Big\| \sim
	\Big\|\sum_k \widetilde{x}_k k^{-s} \Big\|_{\mathcal{H}^{rad}_p (X)} \leq
	C \Big\|\sum_k \widetilde{x}_k k^{-s} \Big\|_{\mathcal{H}_p (X)},
$$	
by definition of the norm in $\mathcal{H}_p(X)$, the last term is equal to
\begin{align*}
	\bigg(\int_{\mathbb{T}^\N} \Big\| \sum_{\alpha\in \N_0^{(\N)}} \widetilde{x}_{n(\alpha)} z^{\alpha} \Big\|^p_X dz\bigg)^{1/p} = 	 \bigg(\int_{\mathbb{T}} \Big\| \sum_{\alpha} x_n z_1^n \Big\|^p_X dz_1\bigg)^{1/p}.
\end{align*}

$(d)\Rightarrow(b)$: We proceed as in \cite[Proposition 2.4]{CaDeSe1}. Recall that by definition of the $\Ha_p (X)$ norm we have
\begin{align*}
	\Big\| \sum_{n=1}^N x_n  n^{-s} \Big\|_{\mathcal{H}_p (X)} = \Big\| \sum_{\alpha\in  \N_0^{(\N)}} x_{n(\alpha)} z^{\alpha} \Big\|_{H_p (X)}.
\end{align*}
 Let $m$ be the maximum of all $\alpha_i$ such that $x_{n(\alpha)}$ is not zero.
For each $z_1 \in \T$ fixed, using a change of variables with $z_i'=z_i\cdot z_1^{(m+1)^{i-1} }$ for $2\leq i\leq N$, we have

\begin{multline*}
	\int_{\T ^{N-1}} \Big\| \sum_{\alpha\in \N_0^{(\N)}} x_{n(\alpha)} z_1^{\alpha_1} z_2^{\alpha_2} \dots z_N^{\alpha_N} \Big\|_X^p dz_2 \dots dz_N \\
	= \int_{\T ^{N-1}} \Big\| \sum_{\alpha\in \N_0^{(\N)}} x_{n(\alpha)} z_1^{\alpha_1} ( z_2 z_1^{m+1})^{\alpha_2} \dots (z_N z_1^{(m+1)^{N-1}})^{\alpha_N} \Big\|_X^p dz_2 \dots dz_N\\
 	= \int_{\T ^{N-1}} \Big\| \sum_{\alpha\in \N_0^{(\N)}} x_{n(\alpha)}z_1^{\alpha_1 + (m+1) \alpha_2 + \dots + (m+1)^{N-1} \alpha_N} z_2^{\alpha_2} \dots z_N^{\alpha_N}\Big\|_X^p dz_2 \dots dz_N. \\
\end{multline*}

Changing the order of integration we get
\begin{multline*}
\int\limits_{\T^N} \Big\| \sum_{\alpha\in  \N_0^{(\N)}} x_{n(\alpha)} z^{\alpha} \Big\|_X^p dz \\
= \int\limits_{\T ^{N-1}} \bigg( \int\limits_{\T} \Big\| \sum_{\alpha\in  \N_0^{(\N)}} (x_{n(\alpha)} z_2^{\alpha_2} \dots z_N^{\alpha_N}) z_1^{\alpha_1 + (m+1) \alpha_2 + \dots + (m+1)^{N-1}
 \alpha_N} \Big\|^p d z_1 \bigg)dz_2 \dots dz_N,
\end{multline*}
Since all the numbers of the form $\sum_{i=1}^N \alpha_i(m+1)^{i-1}$ appearing as powers of $z_1$ are different, we can apply $(d)$ and the contraction principle \eqref{contraction}  in the inner integral with $z_2 , \ldots, z_N$ fixed. We get
\begin{align*}
\int\limits_{\T ^N} \Big\| \sum_{\alpha\in  \N_0^{(\N)}} x_{n(\alpha)} z^{\alpha} \Big\|_X^p dz &\geq \frac{1}{C^p} \int\limits_{\T ^{N-1}} \mathbb{E}  \Big\| \sum_{\alpha\in  \N_0^{(\N)}} (x_{n(\alpha)} z_2^{\alpha_2} \dots z_N^{\alpha_N}) \varepsilon_{\alpha} \Big\|_X^p  dz_2 \dots dz_N \\
&\sim  \mathbb{E} \Big\|\sum_{\alpha\in  \N_0^{(\N)}} x_{n(\alpha)} \varepsilon_{\alpha} \Big\|_X \sim \Big\|\sum_{n=1}^N x_{n} n^{-s}\Big\|_{\mathcal H_p^{rad}(X)},
\end{align*}
where in the last step we used Proposition \ref{rad}.

 $(a)\Rightarrow(b)$: Assume $(b)$ does not hold for any $C\geq1$. Using Proposition \ref{rad}, Lemma \ref{lema} tells us that for any $M\in \N$ and any $K>0$, there is a vector sequence $(x_n)_{n \in \N}\subseteq X$  such that $(x_n n^{-s})_{n\geq M}$ is not $K-$RUC. Using this fact repeatedly for different values of $M$ and $K$, we will construct a vector sequence which contradicts $(a)$. Taking $M=M_0=0$ and $K=1$ we may deduce that there are $M_1\in \N$, $(x_n)_{n=1}^{M_1}\subseteq X$ and $(a_n)_{n=1}^{M_1}\subseteq \C$ such that
$$
\Big\|\sum_{n=1}^{M_1} a_n x_n n^{-s} \Big\|_{\mathcal{H}^{rad}_p (X)} > \Big\|\sum_{n=1}^{M_1} a_n x_n n^{-s} \Big\|_{\mathcal{H}_p (X)}.
$$
Proceeding inductively suppose we have defined $M_k\in \N$, $(x_n)_{n=1}^{M_k}\subseteq X$ and $(a_n)_{n=1}^{M_k}\subseteq \C$ so that
$$
\Big\|\sum_{n=M_{j-1}+1}^{M_j} a_n x_n n^{-s} \Big\|_{\mathcal{H}^{rad}_p (X)} > j\Big\|\sum_{n=M_{j-1}+1}^{M_j} a_n x_n n^{-s} \Big\|_{\mathcal{H}_p (X)},
$$
for every $1\leq j \leq k$.
Taking $M=M_k$ and $K=k+1$ we may deduce that there are $M_{k+1}\in \N$, $(x_n)_{n=M_k+1}^{M_{k+1}}\subseteq X$ and $(a_n)_{n=M_k+1}^{M_{k+1}}\subseteq \C$ such that
$$
\Big\|\sum_{n=M_{k}+1}^{M_{k+1}} a_n x_n n^{-s} \Big\|_{\mathcal{H}^{rad}_p (X)} > (k+1)\Big\|\sum_{n=M_{k}+1}^{M_{k+1}} a_n x_n n^{-s} \Big\|_{\mathcal{H}_p (X)}.
$$
Note that for the sequence $(x_n)_{n}\subseteq X$ thus defined it follows that $(x_n n^{-s})_{n} $ fails to be RUC. Hence, $(a)$ does not hold.
		
$(b)\Rightarrow(c)$: Assuming there is a constant $C>0$ such that \eqref{ec3} holds for Dirichlet polynomials, we prove that the same inequality is valid for every Dirichlet series in $\mathcal{H}_p (X)$. Fix $D\in\mathcal{H}_p (X)$ and $M\in\N$. As mentioned in \ref{coeff}, since Dirichlet polynomials are dense in $\mathcal{H}_p (X)$ there is a sequence of polynomials $(D_N)_{N}\subseteq \mathcal{H}_p (X)$ converging to $D$. In particular, we have that the coefficients $c_n(D_N)$ of $D_N$ converge to those of $D$ for every $n\in\N$. Thus, given $\varepsilon>0$ we may choose $N\in\N$ sufficiently large so that
\begin{align}
 \label{ec1}
 \left\|c_n(D-D_N)\right\|_X&<\frac{\varepsilon}{M},
 \shortintertext{for every $n\leq M$, and}
 \label{ec2}
 \left\|D_N-D\right\|_{\mathcal{H}_p (X)}&<\varepsilon.
\end{align}
Using \eqref{ec1} and the contraction principle \eqref{contraction} we get
\begin{align*}
 \Big\|\sum_{n=1}^{M} c_n(D) n^{-s}\Big\|_{\mathcal{H}^{rad}_p (X)}
 &\leq \Big\|\sum_{n=1}^{M} c_n(D-D_N) n^{-s}\Big\|_{\mathcal{H}^{rad}_p (X)}
 +\Big\|\sum_{n=1}^{M} c_n(D_N) n^{-s}\Big\|_{\mathcal{H}^{rad}_p (X)}
 \\ &\leq \Big\|\sum_{n=1}^{M} c_n(D_N) n^{-s}\Big\|_{\mathcal{H}^{rad}_p (X)} + \varepsilon \leq \left\|D_N\right\|_{\mathcal{H}^{rad}_p (X)} + \varepsilon.
\intertext{From the hypothesis and \eqref{ec2} we obtain}
\Big\|\sum_{n=1}^{M} c_n(D) n^{-s}\Big\|_{\mathcal{H}^{rad}_p (X)}
 &\leq C\left\|D_N\right\|_{\mathcal{H}_p (X)} + \varepsilon
 \leq C\left\|D\right\|_{\mathcal{H}_p (X)} + (C+1)\varepsilon.
\end{align*}
As $\varepsilon$ was arbitrary we have proven that
\begin{align*}
 \Big\|\sum_{n=1}^{M} c_n(D) n^{-s}\Big\|_{\mathcal{H}^{rad}_p (X)}
 \leq C\left\|D\right\|_{\mathcal{H}_p (X)},
\end{align*}
for every $M\in \N$. Finally invoking Corollary \ref{seryestarcoro}, we conclude that
\begin{align*}
 \left\|D\right\|_{\mathcal{H}^{rad}_p (X)}
 \leq C\left\|D\right\|_{\mathcal{H}_p (X)},
\end{align*}
for every $D\in \mathcal{H}_p (X)$.

$(c)\Rightarrow(a)$: This implication is straightforward, so the proof is complete.

\end{proof}

Theorem \ref{EqRCP} brings up a natural question: given a fixed sequence $(x_n)_n$, are conditions \eqref{ec3} and \eqref{ec4} equivalent? In other words, is the sequence $(x_n n^{-s})_n$  $RUC$ in $\Ha_p(X)$ if and only if $(x_n z^n)_n$ is $RUC$ in $H_p (X)$?
In the following examples we provide a negative answer to this question. In fact, we will see that none of the implications hold.

\begin{example}\label{ex:dirichletRUC}\textit{A sequence $(x_n)_n \subseteq X$ such that $(x_n n^{-s})_n$ is $RUC$ in $\Ha_2(X)$ but $(x_n z^n)_n$ fails to be $RUC$ in $H_2 (X)$.}
	
	Let $X$ be the Banach space $L_1(\mathbb{T}^2)$ and consider the sequence $(x_n)_n$ defined by
	$$
	x_n (w_1, w_2)= \left\{
	\begin{array}{c l}
	w_1^n &  \text{ if } n \text{ is prime;}\\
	\quad \\
	w_2^{2^n} & \text{otherwise.}\\
	\end{array}
	\right.
	$$
\end{example}

\begin{proof} First, we see that $(x_n n^{-s})_n$ is $RUC$ in $\Ha_2 (X)$. For $N \in \N$, let $m\in \N$ such that $p_m\leq N< p_{m+1}$, in other words, $\{p_1,\ldots,p_m\}$ is the set of prime numbers less than or equal to $N$. Set $M_N = \{ 1, \ldots, N \} \smallsetminus \{p_1,\ldots,p_m\}$. We have
\begin{align}
\label{ec5}
\mathbb{E}\Big\|\sum_{n=1}^N \varepsilon_n a_n x_n n^{-s}\Big\|_{\Ha_2 (X)} \leq  \mathbb{E}\Big\|\sum_{i=1}^m \varepsilon_{i} a_{p_i} x_{p_i} p_i^{-s}\Big\|_{\Ha_2 (X)} +  \mathbb{E}\Big\|\sum_{n \in M_N} \varepsilon_n a_n x_n n^{-s}\Big\|_{\Ha_2 (X)}.
\end{align}
We analyze both terms in the right hand side separately. First observe that the Bohr transform maps the terms $(p_i^{-s})_{i=1}^m$ to $m$ independent Steinhaus random variables $(z_i)_{i=1}^m$. 
Therefore, using Kahane's inequality \eqref{kahane}, the definition of the norm in $\Ha_2(X)$ and \eqref{contraction2}, we get
\begin{align*}
\mathbb{E}\Big\|\sum_{i=1}^m \varepsilon_{i} a_{p_i} x_{p_i} p_i^{-s}\Big\|_{\Ha_2 (X)} &\sim \bigg(\mathbb{E}\Big\|\sum_{i=1}^m \varepsilon_{i} a_{p_i} x_{p_i} p_i^{-s}\Big\|^2_{\Ha_2 (X)}\bigg)^{\frac12}\\
&=  \bigg(\mathbb{E}\int_{\mathbb T^m} \Big\|\sum_{i=1}^m \varepsilon_{i} a_{p_i} x_{p_i} z_i\Big\|^2_{X}dz_1\ldots dz_m \bigg)^{\frac12}\\
&\lesssim  \bigg(\int_{\mathbb T^m} \Big\|\sum_{i=1}^m a_{p_i} x_{p_i} z_i\Big\|^2_{X}dz_1\ldots dz_m \bigg)^{\frac12}.
\end{align*}
Now, observe that
$$
\int_{\mathbb T} x_n(w_1,w_2)dw_2=\left\{
	\begin{array}{c l}
	w_1^n &  \text{ if } n \text{ is prime;}\\
	\quad \\
	0 & \text{otherwise.}\\
	\end{array}
	\right.
$$
Hence, by Jensen's inequality we have
$$
\Big|\sum_{i=1}^m a_{p_i} w_1^{p_i} z_i\Big| =\Big|\int_{\mathbb T}\sum_{n=1}^N a_{n} x_n(w_1,w_2) z^{\alpha(n)}dw_2\Big| \leq \int_{\mathbb T}\Big|\sum_{n=1}^N a_{n} x_n(w_1,w_2) z^{\alpha(n)}\Big|dw_2.
$$

Therefore, it follows that
\begin{align*}
 \int_{\mathbb T^m} \Big\|\sum_{i=1}^m & a_{p_i} x_{p_i} z_i\Big\|^2_{X}dz_1\ldots dz_m = \int_{\mathbb T^m} \bigg(\int_{\mathbb T}\Big|\sum_{i=1}^m a_{p_i} w_1^{p_i} z_i\Big| dw_1\bigg)^2dz_1\ldots dz_m \\
 &\leq \int_{\mathbb T^m} \bigg(\int_{\mathbb T^2}\Big|\sum_{n=1}^N a_{n} x_n(w_1,w_2) z^{\alpha(n)}\Big| dw_1dw_2\bigg)^2dz_1\ldots dz_m \\
 &=\Big\|\sum_{n=1}^N a_n x_n n^{-s}\Big\|_{\Ha_2 (X)}^2.
\end{align*}

Thus, we have shown that
\begin{equation}\label{ec6}
\mathbb{E}\Big\|\sum_{i=1}^m \varepsilon_{i} a_{p_i} x_{p_i} p_i^{-s}\Big\|_{\Ha_2 (X)}  \lesssim \Big\|\sum_{n=1}^N a_n x_n n^{-s}\Big\|_{\Ha_2 (X)}.
\end{equation}

For the second term in \eqref{ec5}, observe that the variables $w_2^{2^n}$ behave as if they were independent Rademacher variables, since $2^n$ is a lacunary sequence \cite[Theorem 2.1]{pisier1978inegalites}. As it was done before, an unconditionality argument yields
\begin{equation}\label{ec7}
\mathbb{E}\Big\|\sum_{n \in M_N} \varepsilon_n a_n x_n n^{-s}\Big\|_{\Ha_2 (X)} \lesssim \Big\|\sum_{n=1}^N a_n x_n n^{-s}\Big\|_{\Ha_2 (X)}.
\end{equation}
Combining \eqref{ec6} and \eqref{ec7} with \eqref{ec5} we get the desired result.

It remains to see that $(x_n z^n)_n$ fails to be $RUC$ in $H_2 (X)$. The proof of this fact is inspired in that of \cite[Proposition 12.8]{rudin1960trigonometric}, and uses Green-Tao's theorem which states that the sequence of prime numbers contains arbitrarily long arithmetic progressions \cite{TEOTAO}. 

Assume that $(x_n z^n)_n$ is $RUC$. Given $N \in \N$ there exists      an arithmetic progression $A_N$ of length $N$ contained in the prime numbers. Consider the coefficients
$$
a_n = \left\{
\begin{array}{c l}
1 &   \text{if } n \in A_N \\
\quad \\
0 & \text{otherwise.}\\
\end{array}
\right.
$$
Since $(x_n z^n)_n$ is $RUC$ we get

\begin{align*}
\sqrt{N}=
&\mathbb{E}\Big\|\sum_{n \in A_N} \varepsilon_n w_1^n z^n\Big\|_{H_2 (X)}
\lesssim \Big\|\sum_{n \in A_N} w_1^n z^n\Big\|_{H_2 (X)} =  \Big\|\sum_{n \in A_N} w^n\Big\|_{L_1 (\T )} \sim \log N,
\end{align*}
where the $\log N$ term comes from the classical estimation of the $L_1-$norm of the Dirichlet kernel (see for example \cite[pp. 59-60]{mcl}). This leads to a contradiction since the inequality cannot hold for arbitrarily large $N$.
\end{proof}

\begin{example}\label{ex:powerRUC}\textit{A sequence $(x_n)_n \subseteq X$ such that $(x_n z^n)_n$ is $RUC$ in  $H_2 (X)$ but $(x_n n^{-s})_n$ fails to be $RUC$ in $\Ha_2(X)$.}

	Let $X$ be the Banach space $L_1(\mathbb{T}^2)$. Define $(x_n)_n$ by
	$$
	x_n (w_1, w_2)= \left\{
	\begin{array}{c l}
	w_1^k &  \text{ if } n= 2^{k};\\
	\quad \\
	w_2^{2^n} & \text{otherwise.}\\
	\end{array}
	\right.
	$$
\end{example}

\begin{proof} We omit the proof of the first assertion since it is similar to the previous example.
Assume that $(x_n n^{-s})_n$ is $RUC$. In particular, the $RUC$ inequality holds for sequences supported in the powers of two. Recall that the Bohr transform maps $\left(2^k\right)^{-s}$ to $z_1^k$. In other words, we have


\begin{equation}
\mathbb{E}\Big\|\sum_{n=1}^N \varepsilon_{n} a_{n} w_1^n z_1^n\Big\|_{H_2 (X)} \lesssim \Big\|\sum_{n=1}^N  a_{n} w_1^n z_1^n\Big\|_{H_2 (X)}.
\end{equation}
A quick computation leads to  a contradiction since
\begin{equation}
\bigg(\sum_{n=1}^N \left| a_ n\right|^2  \bigg)^{1/2}= \Big\|\sum_{n=1}^N a_n w^n\Big\|_{L_2(\T)} \lesssim \Big\|\sum_{n=1}^N a_n w^n\Big\|_{L_1(\T)}
\end{equation}
cannot hold.
\end{proof}

An analogous result to Proposition \ref{EqRCP} holds replacing the $RUC$ property by $RUD$, leading to the corresponding definition of $\mathcal{H}_p -RDP$. We state this result without proof, as it follows the same arguments.

\begin{proposition}\label{EqRDP}
	Let $X$ be a Banach space and $p \leq 2$. The following statements are equivalent:
	\begin{enumerate}[label=\rm{(\alph*)}]
		\item $(x_n n^{-s})_{n} $ is RUD in $ \mathcal{H}_p (X)$ for every $(x_n)_{n} \subset X$.
		\item There is $C \ge 1$ such that for every $N \in \mathbb{N}$ and $(x_n)^N_{n=1}$ we have
		\begin{align}
		\Big\|\sum_{n=1}^N x_n n^{-s} \Big\|_{\mathcal{H}_p (X)}\leq C \Big\|\sum_{n=1}^N x_n n^{-s} \Big\|_{\mathcal{H}^{rad}_p (X)}.
		\end{align}
		\item The following inclusion holds:
		$$ \mathcal{H}^{rad}_p (X) \subseteq \mathcal{H}_p (X) . $$
		\item There is $C\geq 1$ such that for every $N \in \N$ and $(x_n)^N_{n=1}$ we have
		\begin{equation}
		\bigg( \int_{\mathbb{T}} \Big\| \sum_{n=1}^N x_n z^n \Big\|^p dz \bigg)^{1/p}\leq C \mathbb{E} \Big\| \sum_{n=1}^N \varepsilon_n x_n \Big\| .
		\end{equation}
	\end{enumerate}
\end{proposition}	

\begin{definition}
Given $p \leq 2$, we will say that a Banach space $X$ has the $ \mathcal{H}_p$ \textit{random divergence property} (or, in short, $X$ has $\mathcal{H}_p -RDP$ ) if $X$ satisfies any (and all) of the conditions in Proposition \ref{EqRDP}
\end{definition}

\begin{remark}
	Both examples \ref{ex:dirichletRUC} and \ref{ex:powerRUC} actually work if we consider $x_n \in L_r (\T ^2)$ for $1 < r < 2$ instead of $L_1 (\T ^2)$. Also, the dual statements involving $RUD$ may be proven with the same tools and taking $2< r < \infty$.
\end{remark}

\section{The role of type and cotype}\label{sec-typecotype}

In this section we show that $\mathcal{H}_2 -RCP$ and $\mathcal{H}_2 -RDP$ are equivalent, respectively, to type 2 and cotype 2. The proof is based on the work of Arendt and Bu \cite[Theorem~1.5]{Bu}. We also consider the case $p\ne 2$. For the definition and general properties of type and cotype we refer to \cite[Chapter~11]{DJT}.

\begin{theorem}\label{T2implicaH2RCP} Given a Banach space $X$ the following statements hold:
	\begin{enumerate}[label=\rm{(\roman*)}]
		\item $X$ has type $2$ if and only if it has the $\mathcal{H}_2 -RCP$;
		\item $X$ has cotype $2$ if and only if it has the $\mathcal{H}_2 -RDP$.
	\end{enumerate}
\end{theorem}

Recalling Kwapie\'n's characterization of Hilbert spaces and joining both theorems we arrive at the following conclusion.
\begin{corollary}\label{UniffHilbert}
 For a Banach space $X$ we have that $\Ha_p(X) = \Ha_p^{rad}(X)$ if and only if $p=2$ and $X$ is a Hilbert space.
\end{corollary}
\begin{proof}
	We may deduce from the scalar case that $p$ must be equal to $2$. Furthermore, from Theorem \ref{T2implicaH2RCP} we obtain that $X$ has type and cotype $2$ and therefore it is a Hilbert space. The converse is straightforward.
\end{proof}

Regarding the case where $p>2$, we can apply Theorem \ref{T2implicaH2RCP} to get that type $2$ implies $\Ha_p -RCP$. Whether or not the converse holds still eludes us. However, a slightly weaker result can be established. \begin{theorem}\label{teo3}
	If $X$ has the $\mathcal{H}_p$ random convergence property for some $2 \le p < \infty $, then
	$$\sup\{r:\,X\textrm{ has type }r\}=2.$$
\end{theorem}

For convenience we start by analyzing the  $\Ha_p-RCP$ for the spaces $L_r (\T^{\N})$.

\begin{proposition}
If $ 2 \le r < \infty$, the space $L_r (\mathbb{T}^{\mathbb{N}})$ has the $\mathcal{H}_p -RCP$ for every $2 \le p < \infty $. On the other hand if $ 1\le r < 2$, the space $L_r (\mathbb{T}^{\mathbb{N}})$ does not have the $\mathcal{H}_p -RCP$ for any $2 \le p < \infty $.
\end{proposition}

\begin{proof}
  Assume first that $2 \le r < \infty$. It suffices to show that the space $L_r (\mathbb{T}^{\mathbb{N}})$ enjoys the $\mathcal{H}_2 -RCP$. Indeed, given $(f_n)_{n } \subseteq  L_r (\mathbb{T}^{\mathbb{N}})$ by Kahane's inequality \eqref{kahane}, we have that

 \begin{align*}
	 \Big\|\sum_{n=1}^N  f_n n^{-s}&\Big\|_{\Ha_2 ^{rad} (L_r (\mathbb{T}^{\mathbb{N}}))} \sim \mathbb{E}\Big\|\sum_{n=1}^N \varepsilon_n f_n\Big\|_{L_r (\mathbb{T}^{\mathbb{N}})} \sim \bigg( \int\limits_{\T^ {\N}} \mathbb{E} \Big| \sum_{n=1}^N \varepsilon_n f_n(z) \Big|^r  \,dz \bigg)^{\frac{1}{r}}\\
	 &\sim \bigg( \int\limits_{\T^ {\N}} \bigg( \sum_{n=1}^N \left| f_n(z) \right|^2 \bigg)^\frac{r}{2}  \,dz \bigg)^{\frac{1}{r}} \sim \bigg(  \int\limits_{\T^ {\N}} \Big\|\sum_{n=1}^N f_n (z) n^{-s}\Big\|_{\Ha_2 (\C)} ^r  \,dz \bigg)^{\frac{1}{r}} \\
	 &\le \Big\|\sum_{n=1}^N f_n (z) n^{-s}\Big\|_{\Ha_2 (L_r (\T^{\N}))},
 \end{align*} where in the last step we use Minkowski's integral inequality (regarding the $\mathcal H_2$ norm as an integral via Bohr's transform).

It remains to check that for $1\leq r<2$, the space $L_r ( \mathbb{T}^{\N})$ does not have $ \mathcal{H}_p - RCP$ for any $p \geq 2$.
Let $m\in \mathbb N$ and for every $1\leq n \leq m$ let $f_n \in L_r(\mathbb T^{\mathbb N})$ be the function defined by $f_n(w)=w^{\alpha(n)}$.
Fix scalars $(a_n)_{n=1}^m$. Using Proposition \ref{rad}, the contraction principle \eqref{contraction} and Khintchine's inequality, we get
\begin{align*}
\mathbb E \Big\|\sum_{n=1}^m \varepsilon_n a_n f_n n^{-s} \Big\|_{\mathcal H_{p}(L_r(\mathbb T^{\mathbb N}))} &
\sim \mathbb E \Big\|\sum_{n=1}^m \varepsilon_n a_n f_n \Big\|_{L_r(\mathbb T^{\mathbb N})}
 \gtrsim \mathbb E \Big\|\sum_{n=1}^m \varepsilon_n a_n \Big\|_{L_r(\mathbb T^{\mathbb N})}
 \\ & = \mathbb E \Big|\sum_{n=1}^m \varepsilon_n a_n \Big|
 \gtrsim \bigg( \sum_{n=1}^m \left| a_n \right|^2 \bigg)^{\frac{1}{2}}= \Big\|\sum_{n=1}^m a_n n^{-s} \Big\|_{\mathcal H_2(\mathbb C)}.
\end{align*}
On the other hand, we have
		\begin{align*}
		\Big\|\sum_{n=1}^m a_n f_n n^{-s} \Big\|_{\mathcal H_{p}(L_r(\mathbb
			T^{\mathbb N}))}
		&= \bigg( \int_{T^{\mathbb N}} \Big\|\sum_{n=1}^m a_n f_n z^{\alpha(n)}
		\Big\|_{L_r(\mathbb T^{\mathbb N})}^{p} \,d z\bigg)^{\frac{1}{p}} \\
		&= \bigg( \int_{T^{\mathbb N}} \bigg( \int_{T^{\mathbb N}} \Big|\sum_{n=1}^m
		a_n w^{\alpha(n)} z^{\alpha(n)} \Big|^r \,d w \bigg)^{\frac{p'}{r}} \,d
		z\bigg)^{\frac{1}{p'}} \\
		&= \bigg( \int_{T^{\mathbb N}} \Big|\sum_{n=1}^m a_n w^{\alpha(n)} \Big|^r
		\,d w\bigg)^{\frac{1}{r}}
		= \Big\|\sum_{n=1}^m a_n n^{-s} \Big\|_{\mathcal H_r(\mathbb C)}.
		\end{align*}
Since $r<2$ and $\mathcal{H}_p(\mathbb{C})$ norms are pairwise comparable like $L_p(\mathbb{C})$, there is no constant $C>0$ independent of $m$ such that for every
		choice of scalars $(a_n)_{n=1}^m$
		\begin{align*}
		\Big\|\sum_{n=1}^m a_n n^{-s} \Big\|_{\mathcal H_{2}(\mathbb C)}
		\lesssim \Big\|\sum_{n=1}^m a_n n^{-s} \Big\|_{\mathcal H_{r}(\mathbb
			C)} \quad.
		\end{align*}
		This completes the proof.
	\end{proof}

Now the proof of Theorem \ref{teo3} is simple.

\begin{proof}[Proof of Theorem \ref{teo3}]
We prove this statement by contraposition.
Assume $$s= \sup\{r:\,X\textrm{ has type }r\}<2.$$
By the Maurey-Pisier Theorem \cite{MP} (see also \cite[Chapter~14]{DJT}),
$\ell_s$ is finitely representable
in $X$.
Consequently, the space $L_s(\mathbb T^{\mathbb N})$ is also finitely
representable in $X$ since $L_s$ spaces are finitely representable in $\ell_s$.
As $L_s(\mathbb T^{\mathbb N})$ fails to have the  $\mathcal{H}_p$ random convergence property for $2 \le p < \infty$ and this is clearly a local property, the result follows.
\end{proof}

Finally, we show Theorem \ref{T2implicaH2RCP} holds.

\begin{proof}[Proof of Theorem \ref{T2implicaH2RCP}]
We prove only the first assertion. The second one is omitted since the proof is very similar. Assume that $X$ has type $2$. Let $(\gamma_n)_{n=1}^m$ denote independent identically distributed gaussian random variables.
Given $(x_n)_{n=1}^m\subset X$, let us consider the following operators:
$$
\begin{array}{ccccccccc}
                               T: & \ell_2^m&\rightarrow &X&\hspace{1cm}& S: & L_2(\mathbb T^\N)&\rightarrow &X \\

                                & e_n&\mapsto&x_n& &&z^\alpha&\mapsto&x_{n(\alpha)}
\end{array}
$$
Using the Proposition \ref{rad}, the comparison between Rademacher and gaussian variables (cf. \cite[(4.2)]{T-J}) and \cite[Theorem 12.2]{T-J}, we have
\begin{align}
\label{eq9}
  \mathbb E \Big\|\sum_{n=1}^m \varepsilon_n x_n n^{-s} \Big\|_{\mathcal H_2(X)}\sim \mathbb E \Big\|\sum_{n=1}^m \varepsilon_n x_n\Big\|_X \lesssim \mathbb E \Big\|\sum_{n=1}^m \gamma_n x_n\Big\|_X
  \lesssim \pi_2(T^*),
\end{align}
where $\pi_2$ is the 2-summing operator norm (see \cite[Chapter~2]{DJT} for the definition and basic properties).
From the definition of 2 summing operators, it is easy to check that
\[\pi_2(T^*)=\pi_2(S^*)
\leq \Big\|\sum_{\alpha} x_{n(\alpha)} z^\alpha \Big\|_{\mathcal H_2(X)}
\leq \Big\|\sum_{n=1}^m x_n n^{-s} \Big\|_{\mathcal H_2(X)},\]
which together with \eqref{eq9} proves $X$ has $\mathcal{H}_2 -RCP$.

For the converse we follow \cite{Bu}. Given $x_1, \ldots, x_N\in X$,  we have to prove that
\[
 \bigg( \mathbb{E}_{\varepsilon} \Big\|\sum_{n=1}^N \varepsilon_n x_n \Big\|_X^2 \bigg)^{1/2} \lesssim 	 \bigg( \sum_{n=1}^N \norm{x_n}_X^2 \bigg)^{1/2}.
\]
Fix $f_n\in L^2(\T)$ with $\norm{f_n}_{2}=1$ and disjoint support. An easy calculation gives us
\begin{align*}
 \bigg( \mathbb{E}_{\varepsilon} \int_\T \Big\|\sum_{n=1}^N \varepsilon_n x_n f_n \Big\|^2 \bigg)^{1/2} =\bigg( \sum_{n=1}^N \norm{x_n}_X^2 \bigg)^{1/2}.
\end{align*}
Fix  $\varepsilon >0$. Since trigonometric polynomials are dense in $L_2(\T)$, there is a  polynomial  $h_{n} = \sum_{j= -N_n}^{N_n} a_{n,j} z^{j}$ such that
\[
\|f_n - h_{n}\|_{2} < \frac{\varepsilon}{N\sup_{1\leq i\leq N}\|x_i\|_X} .
\]
There is no loss of generality in assuming that $\norm{h_{n}}_2 =1$.
Therefore, we have
\begin{align}\label{eq5}
\bigg( \mathbb{E}_{\varepsilon} \int_\T \Big\|\sum_{n=1}^N \varepsilon_n h_n x_n\Big\|^2 \bigg)^{1/2} &\le \bigg( \mathbb{E}_{\varepsilon} \int_\T \Big\|\sum_{n=1}^N \varepsilon_n f_n x_n \Big\|^2 \bigg)^{1/2} + \varepsilon \notag\\
&= \bigg( \sum_{n=1}^N \norm{x_n}_X^2 \bigg)^{1/2} + \varepsilon
\end{align}
Notice that we can choose $M_n\in\N$ sufficiently large so that  the powers of $z$ appearing in $h_n z^{M_n}$ are positive and do not overlap. More precisely, we have
\[
 h_n z^{M_n} = \sum_{j \in J_n} b_{n,j} z^j,
\]
where $b_{n,j}= a_{n, j - M_n} $ and $J_n = \{ M_n - N_n , \ldots , M_n + N_n \}$ are pairwise disjoint. By the contraction principle \eqref{contraction} and Proposition \ref{EqRCP} we deduce
\begin{align}	\label{eq6}
	\bigg( \mathbb{E}_{\varepsilon} \int_\T \Big\|\sum_{n=1}^N \varepsilon_n h_n x_n  \Big\|^2 \bigg)^{1/2}
	&\sim \bigg( \mathbb{E}_{\varepsilon} \int_\T \Big\|\sum_{n=1}^N \varepsilon_n  h_n z^{M_n} x_n\Big\|^2 \bigg)^{1/2}\notag
	\\& =\bigg( \mathbb{E}_{\varepsilon} \int_\T \Big\|\sum_{n=1}^N \sum_{j \in J_n} \varepsilon_n  b_{n,j} z^j x_n \Big\|^2 \bigg)^{1/2} \notag
	\\ & \gtrsim \bigg( \mathbb{E}_{\varepsilon} \mathbb{E}_{\delta}\int_\T \Big\|\sum_{n=1}^N \sum_{j \in J_n} \delta_{n,j} \varepsilon_n   b_{n,j} z^j x_n\Big\|^2 \bigg)^{1/2}\notag
	\\ & \sim \bigg( \mathbb{E}_{\delta}  \Big\|\sum_{n=1}^N \sum_{j \in J_n} \delta_{n,j}  b_{n,j} x_n\Big\|^2 \bigg)^{1/2},
\end{align}
where $\delta_n,j$ are independent Bernoulli random variables.  Theorem \ref{teo3} tells us that $X$ has non-trivial type  and therefore finite cotype. Thus,  the variables $\delta_{n,j}$  may be replaced by independent Gaussian variables $\gamma_{n,j}$.  Define $\gamma_n = \sum_{j \in J_n} b_{n,j} \gamma_{n,j}$ and observe that they are independent Gaussian variables of variance 1 since $  \sum_{j \in J_n } |b_{n,j}|^2 = \norm{h_{n} z^{M_n}}_2 = \norm{h_{n}}_2  = 1$. Pushing inequality \eqref{eq6} a little further we get
\begin{align}
\label{eq8}
	\bigg( \mathbb{E}_{\varepsilon} \int_\T \Big\|\sum_{n=1}^N \varepsilon_n  h_n x_n \Big\|^2 \bigg)^{1/2}
	 & \gtrsim \bigg( \mathbb{E}_{\gamma} \Big\|\sum_{n=1}^N \sum_{j \in J_n} b_{n,j} \gamma_{n,j} x_n  \Big\|^2 \bigg)^{1/2}
	 \gtrsim \bigg( \mathbb{E}_{\gamma} \Big\|\sum_{n=1}^N  \gamma_{n} x_n  \Big\|^2 \bigg)^{1/2} \notag
	 \\ & \gtrsim \bigg( \mathbb{E}_{\varepsilon} \Big\|\sum_{n=1}^N  \varepsilon_{n} x_n  \Big\|^2 \bigg)^{1/2}.
\end{align}
Gathering \eqref{eq5} and \eqref{eq8} together leads to the conclusion.
\end{proof}

\subsection*{Acknowledgements} The first three authors were partially supported by CONICET-PIP 11220130100329CO and ANPCyT PICT 2015-2299. The second and third authors are also supported by a CONICET doctoral fellowship. Fourth author gratefully acknowledges support of Spanish Ministerio de Econom\'{\i}a, Industria y Competitividad through grants MTM2016-76808-P, MTM2016-75196-P, and the ``Severo Ochoa Programme for Centres of Excellence in R\&D'' (SEV-2015-0554).

\bibliographystyle{plain}
\bibliography{biblio}
 \bigskip
\noindent {\sc Departamento de Matem\'{a}tica,
	Facultad de Cs. Exactas y Naturales, Universidad de Buenos Aires and IMAS-UBA-CONICET, Argentina.}
 
 \medskip

\noindent \tt dcarando@dm.uba.ar

\noindent \tt fmarceca@dm.uba.ar

\noindent \tt mscotti@dm.uba.ar

\bigskip

\noindent {\sc Instituto de Ciencias Matem\'aticas (CSIC-UAM-UC3M-UCM) 
Consejo Superior de Investigaciones Cient\'ificas 
C/ Nicol\'as Cabrera, 13--15, Campus de Cantoblanco UAM 
28049 Madrid, Spain. }

\medskip

\noindent \tt pedro.tradacete@icmat.es

\end{document}